\documentclass{amsart}

\newcommand{\U}{\mathcal U}
\newcommand{\D}{\mathcal D}
\newcommand{\I}{\mathcal I}

\newcommand{\level}{\operatorname{\ell}}

\newtheorem{theorem}{Theorem}[section]
\newtheorem{lemma}[theorem]{Lemma}
\newtheorem{Algorithm}[theorem]{Algorithm}
\newtheorem{Theorem}[theorem]{Theorem}

\newtheorem{corollary}[theorem]{Corollary}
\newtheorem{proposition}[theorem]{Proposition}
\newtheorem{conjecture}[theorem]{Conjecture}
\theoremstyle{definition}
\newtheorem{definition}[theorem]{Definition}

\usepackage{todonotes}

\usepackage{todonotes}

\usepackage{todonotes}

\usepackage{color}
\usepackage{ulem}
\usepackage{tikz}
\usepackage{tikz-cd}
\usetikzlibrary{arrows}
\usetikzlibrary{tikzmark}
\usepackage{graphicx}
\usepackage{caption}

\title[Proof of a conjecture of Matherne, Morales, and Selover]{Proof of a conjecture of Matherne, Morales, and Selover on encodings of unit interval orders}

\author{Félix Gélinas, Adrien Segovia, and Hugh Thomas}

\begin{document}
\maketitle
\normalem 

\begin{abstract}
    There are two bijections from unit interval orders on $n$ elements to Dyck paths from $(0,0)$ to $(n,n)$. One is to consider the pairs of incomparable elements, which form the set of boxes between some Dyck path and the diagonal. Another is to find a particular part listing (in the sense of Guay-Paquet) which yields an isomorphic poset, and to interpret the part listing as the area sequence of a Dyck path. Matherne, Morales, and Selover conjectured that, for any unit interval order, these two Dyck paths are related by Haglund's well-known zeta bijection. In this paper we prove their conjecture.  
\end{abstract}

\section{Introduction}\label{intro}

Let $\mathcal I=\{I_1,\dots,I_n\}$
be a set of $n$ intervals of unit length, numbered
  from left to right. $U(\I)$ is the poset on the elements
  1 to $n$ which is defined by $i\prec j$ if and only if interval
  $I_i$ is strictly to the left of interval $I_j$. 

The posets that can be described in this way are called \emph{unit interval orders}. We write $\U_n$ for the collection of unit interval orders on
  $\{1,\dots,n\}$.

 A \emph{Dyck path} of length $n$ is a path which
  proceeds by steps of $(1,0)$ and $(0,1)$ from $(0,0)$ to $(n,n)$, while
  not passing below the diagonal line from $(0,0)$ to $(n,n)$. Dyck paths are
  counted by the well-known Catalan numbers.  We write $\D_n$ for the collection of Dyck paths of length $n$.

  A Dyck path can be specified by identifying the collection of complete
  $1\times 1$ boxes with vertices at lattice points which lie between the Dyck path and the diagonal line. We refer to the collection of these boxes as the
  \emph{area set} of the Dyck path. 
  The boxes which may appear
  in the area set of a Dyck path in $\D_n$ can be indexed by pairs $(i,j)$ with
  $1\leq i <j \leq n$, where the pair $(i,j)$ corresponds to the box $i-1\leq x \leq i$, $j-1\leq y\leq j$. 

  The \emph{area sequence} of a Dyck path $D\in\D_n$ is the
  sequence $(a_1,\dots,a_n)$ consisting of the number
  of boxes
  in the area set on each horizontal line, from bottom to top. The possible area sequences of
  Dyck paths are characterized by the fact that $a_1=0$ and
  $a_i\leq a_{i-1}+1$. 
 
  For $U$ a unit interval order, let $a(U)$ be the Dyck path whose area set
  is given by the boxes $(i,j)$ such that $i<j$ and $i\not\prec j$.
  In Section \ref{Area}, we will recall the proof that this
  defines a bijection from $\U_n$ to $\D_n$.

  We now turn to the definition of a second map from unit interval orders
  to Dyck paths, as described in Section \ref{PL}.
  Let $w=(w_1,\dots,w_n)$ be a sequence of non-negative integers.
Associated to $w$ is a poset $P(w)$ defined as follows \cite[Section 2]{GP}.

  For $1\leq i,j\leq n$, we set $i\prec j$ if either of the following is satisfied: \begin{itemize}
  \item   $w_j -w_i\geq 2$,
  \item  or $w_j-w_i =1$ and $i<j$.
\end{itemize}

  For any $w$, the poset $P(w)$ is isomorphic to a unique unit interval order.
The word $w$ is referred to as a \emph{part listing} for $P(w)$. 
  Note that the labelling of the elements of $P(w)$ and the labelling of
  the elements of the isomorphic unit interval order may not be the same.

For a unit interval order $U\in \U_n$, there is a unique part listing $w$
such that $P(w)$ is isomorphic to $U$ and $w$ is the area sequence of a
Dyck path. Define $\tilde p(U)$ to be this part listing.
Define $p(U)$ to be the Dyck path whose area sequence is $\tilde p(U)$.

We have now recalled two ways to 
associate a Dyck path to a unit interval
order, namely the maps $a$ and $p$. Considering these two maps,
Matherne, Morales and Selover raised this question of how they are
related. (In fact, the definition of $p$ used in \cite{MMS} is slightly different from what we have taken as the definition, and seems to include a small inaccuracy. The authors of \cite{MMS} have confirmed to us that the map $p$ which we consider is the one which they intended.)

On the basis of computer evidence, Matherne, Morales, and Selover made the following conjecture:

\begin{conjecture}\cite{MMS} \label{con}
  For $U\in \U_n$, we have that 
  $a(U)=\zeta(p(U))$.\end{conjecture}

The map $\zeta$ which relates $a(U)$ and $p(U)$ in the conjecture of Matherne,
Morales, and Selover is the famous zeta map of Haglund \cite{H}.
It is a bijection from $\D_n$ to $\D_n$ which plays
an important rôle in $(q,t)$-Catalan combinatorics. We recall its
definition in Section \ref{Zeta} below. This 
unexpected connection between $\zeta$ and unit interval orders seems worthy of further investigation.

The main result of this paper is to prove Conjecture \ref{con} (Theorem \ref{theoremConj}) which we carry
out in Section \ref{proof}.

Additionally, in Section \ref{grevlex}, we establish that the part listing $\tilde p(U)$ can be characterized among those part listings $w$ with
$P(w)$ isomorphic to
$U$ as the one which is minimal with respect to graded reverse lexicographic
order. This is not needed for our argument, but provides an interesting alternative description of the map $\tilde p$ (and thus also of $p$).

\section{The map $a$}\label{Area}

Let us give an equivalent definition of the unit interval orders.

\begin{lemma}
\label{defequivalent}

  An order $\prec$ on $\{1,\dots,n\}$ is a unit interval order if and only if for all $x,y\in \{1,\dots,n\}$, we have the two properties
  \begin{itemize}
      \item  $x\prec y$ implies $x<y$ ,
      \item  $x\prec y$ implies $\forall x'\leq x$ and $\forall y'\geq y$, we have $x'\prec y'$.
  \end{itemize}
\end{lemma}

\begin{proof}
  It is obvious from the definition that a unit interval order satisfies the two properties. We prove the converse direction by contradiction. 
  It is clearly true for $n=1$. Suppose that $\prec$ is an order on $\{1,\dots,n \}$, with $n\geq 2$ minimal, that satisfies the two properties but such that it is not a unit interval order. The two properties are such that when $\prec$ is restricted to $\{1,\dots,n-1\}$, it still satisfies the two properties. By minimality of $n$, the order $\prec$ restricted to $\{1,\dots,n-1\}$ is a unit interval order. Let $\{I_1,\dots,I_{n-1}\}$ be a set of intervals of unit length realizing this unit interval order. 
  
If the integer $n$ is not comparable to any other interval with respect to $\prec$, then the second property implies that there are no relations for $\prec$. The poset $\prec$ can therefore be realized as the unit interval order corresponding to $n$ intervals all of which overlap. This contradicts our assumption.

  Otherwise let $k$ be the greatest integer such that $k\prec n$. By the second property, this means that 
  the intervals $I_i$ for $k<i<n$ (if any) intersect. We can therefore choose $I_n$ so that it intersects exactly these interval and no others.
  Using the second property, the unit interval order defined by $\{I_1,\dots,I_n\}$ is precisely $\prec$, which is a contradiction and finishes the proof of the lemma.
\end{proof}

  Given $U\in \U_n$, define
  $$\tilde a(U) = \{(x,y)\mid x\not\prec y, 1\leq x<y \leq n\} $$

  \begin{lemma} \label{area} For $U\in \U_n$, we have that
    $\tilde a(U)$ is the area set of a Dyck path in $\D_n$. \end{lemma}
    
  \begin{proof}

    Since $(x,y)\in\tilde a(U)$ implies $x<y$, we only have boxes above the diagonal. 
    Let $(x,y)\in\tilde a(U)$. Then for any $x\leq x'<y'\leq y$, we have
    that $(x',y')\in\tilde a(U)$. Indeed, since the intervals $I_x$ and $I_y$ overlap,
    so do all the intervals indexed by numbers between $x$ and $y$. 
    We just proved that if there is a box $(x,y)$ in $\tilde a(U)$, all the boxes above the diagonal and weakly south east of $(x,y)$ are also in $\tilde a(U)$. This implies that $\tilde a(U)$ is the area set of a Dyck path.
  \end{proof}

  Thanks to Lemma \ref{area}, for $U$ a unit interval order, we can define
  $a(U)$ to be the Dyck path whose area set is given by $\tilde a(U)$.

  \begin{lemma} The map $a$ is a bijection from $\U_n$ to $\D_n$. \end{lemma}
  
    \begin{proof}
    
   We give the inverse map. Let $D\in \D_n$ be a Dyck path with area set $A$. Let $\prec$ be the order on $\{1,\dots,n\}$ defined by $x\prec y$ if $x<y$ and $(x,y)\not\in A$. Proving that $\prec$ is a unit interval order would finish the proof, as the map sending $D$ to $\prec$ would be the inverse of $a$. Since $A$ is the area set of a Dyck path we know that if $(x,y)\not\in A$, then $\forall x'\leq x$ and $\forall y'\geq y$, we have that $(x',y')\not\in A$. The order $\prec$ is therefore a unit interval order by Lemma \ref{defequivalent}. \end{proof}

\section{Part listings} \label{PL}

In this section, we study the way that unit interval orders can be
defined via part listings, as already described in
Section \ref{intro}. We begin by giving an algorithm which, starting from
a unit interval order $U\in\U_n$, gives the part listing $\tilde p(U)$ defined in the introduction as the unique part listing whose associated poset is isomorphic to $U$ and such that it is the area sequence of a Dyck path.

Given $U$ a unit interval order in $\U_n$, we inductively define a function $\level$ from $\{1,\dots,n\}$ to $\mathbb Z_{\geq 0}$, as follows. We fix that
$\level(1)=0$. We suppose that $\level(i)$ has been defined for all
$1\leq i \leq j-1$. Now define $\level(j)$ to be $\max_{i\prec j} \level(i)+1$.
If $j$ is a minimal element of $U$, so that the set over which we are taking
the maximum is empty, we define $\level(j)=0$. We call $\ell(i)$ the level of $i$ (or of the interval $I_i$).

    \begin{Algorithm} \label{algo}   

      Let $U\in \U_n$. We will successively define words $q_1$, $q_2$, \dots,
      $q_n$. The word $q_i$ is of length $i$, and is obtained by inserting
      a copy of $\ell(i)$ into $q_{i-1}$.

      \begin{itemize}
          \item We begin by defining $q_1=0$. Now suppose that $q_{i-1}$ has
      already been constructed.

        \item Let $C_i$ be the number of elements of level $\ell(i)-1$ comparable to $i$.(Note that they are necessarily to the left of $i$.)
        The letter $\ell(i)$ is added into $q_{i-1}$ directly after the occurrences of the letter $\ell(i)$ (if any) immediately following the $C_i$-th letter $\ell(i)-1$.
        
    \end{itemize}
    Finally, define $q(U)=q_n$.
    \end{Algorithm}

See Figure \ref{fig:algo} for an example of this algorithm.

\begin{figure}
    \centering
    
    \begin{tikzpicture}[auto,scale=0.95]
    

    \node (ar3) at (6,2)
    {\begin{tabular}{lll}
    $\ell(1)=0$ & $C_1=0$ & 
 {$q_1=({\color{blue} 0})$}\\
 {$\ell(2)=0$}&
$C_2=0$&
$q_2=(0,{\color{blue} 0})$\\
   {$\ell(3)=1$}&
    $C_3=1$&
    $q_3=(0,{\color{blue} 1},0)$\\
     {$\ell(4)=1$}&{$C_4=1$}&
 {$q_4=(0,1,{\color{blue} 1},0)$}\\ {$\ell(5)=2$}&
{$C_5=1$}&
{$q_5=(0,1,{\color{blue} 2},1,0)$}\end{tabular}};

    
    \node (uio5) at (0,2) {\tikzpicture
    \draw[thick] (0,0) -- (1.5,0) node[xshift=-0.75cm,fill=white] {\begin{tikzpicture} \node (a) {\scriptsize{1}}; \end{tikzpicture}};
    \draw (0,-0.1) -- (0,0.1);
    \draw (1.5,-0.1) -- (1.5,0.1);
    
    \draw[thick] (1,0.4) -- (2.5,0.4) node[xshift=-0.75cm,fill=white] {\begin{tikzpicture} \node (a) {\scriptsize{2}}; \end{tikzpicture}};
    \draw (1,0.3) -- (1,0.5);
    \draw (2.5,0.3) -- (2.5,0.5);
    
    \draw[thick] (1.75,0) -- (3.25,0) node[xshift=-0.75cm,fill=white] {\begin{tikzpicture} \node (a) {\scriptsize{3}}; \end{tikzpicture}};
    \draw (1.75,-0.1) -- (1.75,0.1);
    \draw (3.25,-0.1) -- (3.25,0.1);
    
    \draw[thick] (2.25,0.8) -- (3.75,0.8) node[xshift=-0.75cm,fill=white] {\begin{tikzpicture} \node (a) {\scriptsize{4}}; \end{tikzpicture}};
    \draw (2.25,0.7) -- (2.25,0.9);
    \draw (3.75,0.7) -- (3.75,0.9);
    
    \draw[thick] (3.5,0.4) -- (5,0.4) node[xshift=-0.75cm,fill=white] {\begin{tikzpicture} \node (a) {\scriptsize{$5$}}; \end{tikzpicture}};
    \draw (3.5,0.3) -- (3.5,0.5);
    \draw (5,0.3) -- (5,0.5);
    
    \endtikzpicture};

\end{tikzpicture}

    \caption{Example of Algorithm \ref{algo}}
    \label{fig:algo}
\end{figure}
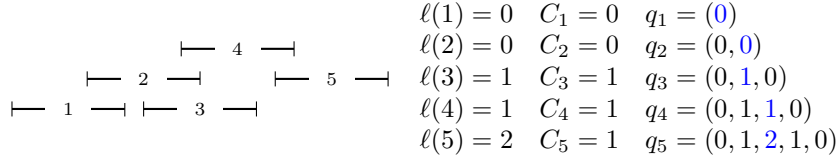

\begin{lemma} \label{asequ} For $U\in \U_n$, we have that $q(U)$ is a part listing corresponding to the area sequence of a Dyck path in $\D_n$. \end{lemma}

\begin{proof} 

Recall that the area sequences of Dyck paths in $\D_n$ are the sequences of non-negative integers $(a_1,\dots,a_n)$ characterized by the two properties that $a_1=0$ and $a_i \leq a_{i-1}+1$.

Let $U$ be a unit interval order and $q(U)$ be the sequence constructed by the above algorithm. By definition of the algorithm, $0$ is the first element added to the list and will remain at the first position in $q(U)$ since every level $\ell(i)$ will be added after this $0$. Thus, we have that $q(U)$ starts with a $0$. 
Also by construction, we have that $\ell(i)$ is inserted directly after a letter which is either $\ell(i)-1$ or $\ell(i)$. Thus, when $\ell(i)$ is inserted, it satisfies the condition that its value is at most one more than the value of its preceding letter. And since at each step we add letters greater or equal to the maximal letter of the previous word, the condition will remain true as we carry out
all subsequent insertions.
\end{proof}

Given a part listing $w$ of length $n$, we define a partial order on
$\{1,\dots,n\}$ as described in the introduction. We denote it by $P(w)$, and we write 
$\prec_{P(w)}$ for its order relation.

\begin{definition} 
    Given a unit interval order $U\in \U_n$, and its corresponding part listing
    $w=q(U)$, we define a permutation
    $f_U:\{1,\dots,n\}\rightarrow \{1,\dots,n\}$ as follows. We first
    consider the positions $i$ in the part listing with $w_i=0$, and we
    number them starting with 1 from left to right. We then number the positions $i$ with $w_i=1$ from left to right, and continue in the same way. The
    number that is eventually assigned to position $i$ is $f(i)$.
    \end{definition}

We now define a new order $\prec_f$ on $\{1,\dots,n\}$ by saying that
$i\prec_f j$ if and only if $f^{-1}(i) \prec_{P(w)} f^{-1}(j)$. By definition,
this poset is isomorphic to $P(w)$.

\begin{proposition}\label{iso} For $U$ a unit interval order, $w=q(U)$ the corresponding part
  listing, and $f$ the bijection defined above, $\prec_f$ agrees with the
  original order on $U$. \end{proposition}

\begin{proof} Note that $\ell$ is a weakly increasing function of $\{1,\dots,n\}$. 
  Let $1\leq i\leq n$. Because the copies of $\ell$ in $w$ are
  inserted from left to right, $f^{-1}(i)$ is the position in $w$ where
  $\ell(i)$ wound up, that is to say, $w_{f^{-1}(i)}=\ell(i)$. 

  Let $1\leq i<j\leq n$. Suppose $I_i$ is strictly to the left of $I_j$ in $U$.
  Thus, $\ell(j)>\ell(i)$. Suppose first that $\ell(j) \geq \ell(i)+2$.
  In this case, $w_{f^{-1}(j)} \geq w_{f^{-1}(i)}+2$, so $j \succ_f i$, as desired.

  Suppose now that $\ell(j)=\ell(i)+1$. Suppose that $I_j$ is strictly to
  the right of $C$ intervals of level $\ell(i)$. Note that $I_i$ is one of
  them by assumption. Thus, when $\ell(j)$ was inserted into $w$, it was inserted to the
  right of the letter $\ell(i)$; this persists as other letters are inserted.
  It follows that also in this case, $j\succ_f i$, as desired.

  Suppose now that $I_i$ and $I_j$ overlap. In this case, $\ell(j)\leq \ell(i)+1$, because the $I_k$ strictly to the left of $I_j$ are weakly to the left of
  $I_i$, so have level at most $\ell(i)$. We must therefore consider two cases,
  when $\ell(j)=\ell(i)$ and when $\ell(j)=\ell(i)+1$.

  Consider first the case where $\ell(j)=\ell(i)$. In this case,
  $w_{f^{-1}(j)}=w_{f^{-1}(i)}$, so $i$ and $j$ are incomparable with respect to
  $\prec_f$, as desired.

  The case where $\ell(j)=\ell(i)+1$ is disposed of similarly to the
  case $\ell(j)=\ell(i)+1$ where $I_i$ and $I_j$ do not overlap; in this
  case, the result is that $f^{-1}(j)<f^{-1}(i)$, with the result that
  $i$ and $j$ are incomparable with respect to $\prec_f$, as desired.
  This completes the proof.
  \end{proof}

In the introduction, for $U$ a unit interval order, we defined 
$\tilde p(U)$ to be the unique part listing $w$ such that $P(w)$ is isomorphic to $U$ and $w$ is also an area sequence of a Dyck path. We are 
now in a position to establish that 
the map $\tilde p$ is well-defined.

\begin{proposition}
For $U$ a unit interval order, we have $\tilde p(U)$ is well-defined and $\tilde p(U)=q(U)$.
\end{proposition}

\begin{proof} 
For each unit interval order $U\in \U_n$, the part listing $q(U)$ is 
the area sequence of a Dyck path
by Lemma \ref{asequ}. 
By Proposition \ref{iso}, the poset $P(q(U))$ is isomorphic to $U$. This shows in particular that the map $q$ must be injective. Since there are the same number of unit interval orders in $\U_n$ as of Dyck paths in $\D_n$, $q$ must be a bijection. Thus, Proposition \ref{iso} tells us that if $w$ and $w'$ are two different area sequences of Dyck paths, then $P(w)$ and $P(w')$ cannot be isomorphic. It follows that, for any $U$, there is exactly one Dyck path $w$ such that $P(w)$ is isomorphic to $U$, namely, $q(U)$. Thus, $\tilde p(U)$ is well-defined and equals $q(U)$.
\end{proof}

\section{The zeta map} \label{Zeta}
 We now describe the map $\zeta:\D_n\rightarrow \D_n$.
 Start with $D\in \D_n$. We begin by labelling
 the lattice points that make up the path $D$ (except the very first):
 we label the top end-point of an up step with the letter $a$, and we label
 the right endpoint of a right step with the letter $b$.

We then read the labels: first on the line $y=x$, from bottom left to top right,
 then on the line $y=x+1$, again in the same direction, then on the line
 $y=x+2$, etc. Interpret $b$ as designating an up step, and $a$ as
 designating a right step. This defines a lattice path from $(0,0)$ to
 $(n,n)$. Define this to be $\zeta(D)$. See Figure \ref{fig:zeta} for an example of this map.

  \begin{lemma}\label{app} Starting from $D\in\D_n$, the path $\zeta(D)$ is a
    Dyck path. \end{lemma}

\begin{figure}
    \centering
    
    \begin{tikzpicture}[auto]
 
 \node[label={[xshift=0cm, yshift=-6.3cm] $w=a\,a\,a\,b\,a\,b\,a\,b\,b\,b\,a\,b$}] (d) at (1,0) {\begin{tikzpicture}[scale=0.8]
    \fill(2cm,1.5cm); 
    \draw (0,0) grid (6,6);
    \draw[dashed,thin] (-0.5,-0.5)--(6.5,6.5);
    \draw[dashed,thin] (-0.5,0.5)--(5.5,6.5);
    \draw[dashed,thin] (-0.5,1.5)--(4.5,6.5);
    \draw[dashed,thin] (-0.5,2.5)--(3.5,6.5);
    \draw[blue,line width=0.8mm] (0,0)--(0,1);
    \draw[blue,line width=0.8mm] (0,1)--(0,2);
    \draw[blue,line width=0.8mm] (0,2)--(0,3);
    \draw[blue,line width=0.8mm] (0,3)--(1,3);
    \draw[blue,line width=0.8mm] (1,3)--(1,4);
    \draw[blue,line width=0.8mm] (1,4)--(2,4);
    \draw[blue,line width=0.8mm] (2,4)--(2,5);
    \draw[blue,line width=0.8mm] (2,5)--(3,5);
    \draw[blue,line width=0.8mm] (3,5)--(4,5);
    \draw[blue,line width=0.8mm] (4,5)--(5,5);
    \draw[blue,line width=0.8mm] (5,5)--(5,6);
    \draw[blue,line width=0.8mm] (5,6)--(6,6);
    \draw (0,1) node [left]{$a$};
    \draw (0,2) node [left]{$a$};
    \draw (0,3) node [left]{$a$};
    \draw (1,3) node [above left]{$b$};
    \draw (1,4) node [above left]{$a$};
    \draw (2,4) node [above left]{$b$};
    \draw (2,5) node [above left]{$a$};
    \draw (3,5) node [above left]{$b$};
    \draw (4,5) node [above left]{$b$};
    \draw (5,5) node [above left]{$b$};
    \draw (5,6) node [above left]{$a$};
    \draw (6,6) node [above]{$b$};
 \end{tikzpicture}};
 
 \node[label={[xshift=0cm, yshift=-6.3cm] $\zeta(w) = a\,a\,b\,a\,b\,b\,a\,a\,a\,b\,b\,b$}] (c) at (8,0) {\begin{tikzpicture}[scale=0.8]
    \fill(2cm,1.5cm); 
    \draw (0,0) grid (6,6);
    \draw[dashed, thin] (-0.5,-0.5)--(6.5,6.5);
    \draw[blue,line width=0.8mm] (0,0)--(0,1);
    \draw[blue,line width=0.8mm] (0,1)--(0,2);
    \draw[blue,line width=0.8mm] (0,2)--(1,2);
    \draw[blue,line width=0.8mm] (1,2)--(1,3);
    \draw[blue,line width=0.8mm] (1,3)--(2,3);
    \draw[blue,line width=0.8mm] (2,3)--(3,3);
    \draw[blue,line width=0.8mm] (3,3)--(3,4);
    \draw[blue,line width=0.8mm] (3,4)--(3,5);
    \draw[blue,line width=0.8mm] (3,5)--(3,6);
    \draw[blue,line width=0.8mm] (3,6)--(4,6);
    \draw[blue,line width=0.8mm] (4,6)--(5,6);
    \draw[blue,line width=0.8mm] (5,6)--(6,6);
    \draw (0,1) node [left]{$a$};
    \draw (0,2) node [left]{$a$};
    \draw (1,2) node [above left]{$b$};
    \draw (1,3) node [above left]{$a$};
    \draw (2,3) node [above left]{$b$};
    \draw (3,3) node [above left]{$b$};
    \draw (3,4) node [above left]{$a$};
    \draw (3,5) node [above left]{$a$};
    \draw (3,6) node [above]{$a$};
    \draw (4,6) node [above]{$b$};
    \draw (5,6) node [above]{$b$};
    \draw (6,6) node [above]{$b$};
 \end{tikzpicture}} ;
 \draw[->] (d) to node {$\zeta$} (c);
\end{tikzpicture}

    \caption{Example of the zeta map }
    \label{fig:zeta}
\end{figure}
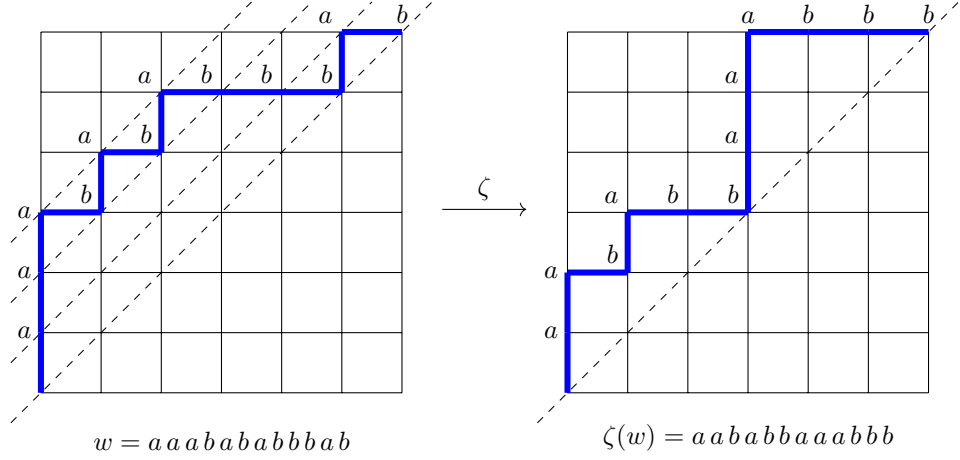

   \begin{proof}
   Let $D\in\D_n$.
  We define a matching between the up steps and the right steps of $D$ as follows. For all non-negative integers $t$, look at the part of $D$ between the lines $y=x+t$ and $y=x+(t+1)$. This necessarily consists of an alternating sequence of the same number of up steps and right steps. We define our pairing by matching the $i$-{th} up step with the $i$-{th} right step. By definition of $\zeta$, these two matched edges contribute an up step and a right step to $\zeta(D)$, and the up step comes before the right step. Thus $\zeta(D)$ will always stay above the diagonal. 
\end{proof}

\section{Proof of the conjecture}\label{proof}

The proof of the conjecture (Theorem \ref{theoremConj}) will proceed by induction. We suppose that for a unit interval order $U$, we know that $\zeta(p(U))$ and $a(U)$ coincide. We then consider what happens when we add a new rightmost interval to $U$. Proving that this changes the result of applying each of the maps in the same way, we conclude that the two maps also coincide on the larger poset. This proves the conjecture by induction.

\begin{definition}

   A peak in a Dyck path consists of an up step followed by a
    right step. In the Dyck word, this amounts to an occurrence of the
    consecutive pair of letters `$ab$' and in the corresponding area sequence $(a_1,\dots,a_n)$ this amounts to have $a_{i}\leq a_{i-1}$ for a given $i$.

    By adding a peak to a Dyck path, we mean the insertion,
    at some position, of `$ab$' into the Dyck path. The result of adding a peak is
    again a Dyck path. Note that adding a peak does not necessarily
    increase the number of peaks: if the peak is added after a letter `$a$' and
  before a letter `$b$', the number of peaks does not change.

   We say that a peak of a Dyck path is a maximal peak if the top of the up step lies on the highest line of slope $1$ that touches the Dyck path. We will refer to the last peak of the Dyck path as the final peak, and to the last maximal peak as the final maximal peak.

  We say that a peak is of height $i$ if the top of the up step is on the line
  $y=x+i$.
\end{definition}

 In the sequel, let $U\in \U_n$ and $U'$ be obtained from
  $U$ by adding an $(n+1)$-st interval of unit length to the right of those
  of $U$. We note $\ell$ the level of this added interval.

 For an illustration of the following lemmas, see Figure \ref{fig:expan}.

\begin{lemma}\label{add}  Let $U$ be a unit interval order, and let $U'$ be obtained from it by adding an interval to the right of the intervals of $U$. 
The Dyck path $p(U')$ is obtained from $p(U)$ by adding a final maximal peak.
\end{lemma}

\begin{proof}
  By Algorithm \ref{algo}, in going from $p(U)$ to $p(U')$, we insert a letter $\ell$ into $p(U)$ such that the letter before is either $\ell$ or $\ell-1$. The letters before are at most $\ell$ and the letters
  after are at most $\ell-1$. Adding a letter $\ell$ in this way
  adds a final maximal peak, since we add the rightmost maximal letter $\ell$ to the area sequence $p(U)$.
  \end{proof}
  
\begin{figure}
    \centering

\begin{tikzpicture}[scale=0.8]

    \node (uio5) at (0,0) {\tikzpicture
    \draw[thick] (0,0) -- (1.5,0) node[xshift=-0.75cm,fill=white] {\begin{tikzpicture} \node (a) {\scriptsize{1}}; \end{tikzpicture}};
    \draw (0,-0.1) -- (0,0.1);
    \draw (1.5,-0.1) -- (1.5,0.1);
    
    
    \draw[thick] (1.75,0) -- (3.25,0) node[xshift=-0.75cm,fill=white] {\begin{tikzpicture} \node (a) {\scriptsize{2}}; \end{tikzpicture}};
    \draw (1.75,-0.1) -- (1.75,0.1);
    \draw (3.25,-0.1) -- (3.25,0.1);
    
    \draw[thick] (2.75,0.4) -- (4.25,0.4) node[xshift=-0.75cm,fill=white] {\begin{tikzpicture} \node (a) {\scriptsize{3}}; \end{tikzpicture}};
    \draw (2.75,0.3) -- (2.75,0.5);
    \draw (4.25,0.3) -- (4.25,0.5);
    
    \draw[thick] (3.5,0) -- (5,0) node[xshift=-0.75cm,fill=white] {\begin{tikzpicture} \node (a) {\scriptsize{4}}; \end{tikzpicture}};
    \draw (3.5,0.1) -- (3.5,-0.1);
    \draw (5,0.1) -- (5,-0.1);
    
    \endtikzpicture};

  \draw[->] (4,0) -- (5,0);

  \node (uio6) at (9.5,0.2) {\tikzpicture
    \draw[thick] (0,0) -- (1.5,0) node[xshift=-0.75cm,fill=white] {\begin{tikzpicture} \node (a) {\scriptsize{1}}; \end{tikzpicture}};
    \draw (0,-0.1) -- (0,0.1);
    \draw (1.5,-0.1) -- (1.5,0.1);
    
    \draw[thick] (3.87,0.8) -- (5.37,0.8) node[xshift=-0.75cm,fill=white] {\begin{tikzpicture} \node (a) {\scriptsize{5}}; \end{tikzpicture}};
    \draw (3.87,0.7) -- (3.87,0.9);
    \draw (5.37,0.7) -- (5.37,0.9);
    
    \draw[thick] (1.75,0) -- (3.25,0) node[xshift=-0.75cm,fill=white] {\begin{tikzpicture} \node (a) {\scriptsize{2}}; \end{tikzpicture}};
    \draw (1.75,-0.1) -- (1.75,0.1);
    \draw (3.25,-0.1) -- (3.25,0.1);
    
    \draw[thick] (2.75,0.4) -- (4.25,0.4) node[xshift=-0.75cm,fill=white] {\begin{tikzpicture} \node (a) {\scriptsize{3}}; \end{tikzpicture}};
    \draw (2.75,0.3) -- (2.75,0.5);
    \draw (4.25,0.3) -- (4.25,0.5);
    
    \draw[thick] (3.5,0) -- (5,0) node[xshift=-0.75cm,fill=white] {\begin{tikzpicture} \node (a) {\scriptsize{4}}; \end{tikzpicture}};
    \draw (3.5,0.1) -- (3.5,-0.1);
    \draw (5,0.1) -- (5,-0.1);
    
    \endtikzpicture};

    \draw (0.5,-1.2) node{$U$};
    \draw (10,-1.2) node{$U'$};

\begin{scope}[xshift= -1.5cm, yshift=-8cm]
   \fill(2cm,1.5cm); 
    \draw (0,0) grid (4,4);
    \draw[dashed,thin] (-0.5,-0.5)--(4.5,4.5);
    \draw[blue,line width=0.6mm] (0,0)--(0,3)--(2,3)--(2,4)--(4,4);
    
    
    \draw (0,1) node[left]{$a$};
    \draw (0,2) node[left]{$a$};
    \draw (0,3) node[left]{$a$};
    \draw (1,3) node[above left]{$b$};
    \draw (2,3) node[above left]{$b$};
    \draw (2,4) node[above]{$a$};
    \draw (3,4) node[above]{$b$};
    \draw (4,4) node[above]{$b$};

\draw (2,-1.5) node{$p(U)=(0,1,2,1)$};
\draw (2,-2.2) node{$a\,a\,a\,b\,b\,a\,b\,b$};

\draw[->] (5.5,2)--(6.5,2);

\begin{scope}[scale=1, xshift= 8cm, yshift= -0.5cm]
    \fill(2cm,1.5cm); 
    \draw (0,0) grid (5,5);
    \draw[dashed,thin] (-0.5,-0.5)--(5.5,5.5);
    \draw[blue,line width=0.6mm] (0,0)--(0,3)--(1,3);
    \draw[blue,line width=0.6mm] (2,4)--(3,4)--(3,5)--(5,5);
    
    \draw[red,line width=0.6mm] 
    (1,3)--(1,4)--(2,4);
    
    
    \draw (0,1) node [left]{$a$};
    
    
    \draw (0,2) node [left]{$a$};
    
    
    \draw (0,3) node [left]{$a$};
    
    \draw (1,4) node{$\textcolor{red}{\bullet}$};
    
    \draw (1,3) node [above left]{$b$};
    
    \draw (2,4) node{$\textcolor{red}{\bullet}$};
    
    \draw (1,4) node [above left, red]{$a$};
    
    
    \draw (2,4) node [above left, red]{$b$};
    
    
    \draw (3,4) node [above left]{$b$};
    
    
    \draw (3,5) node[above]{$a$};
    
    
    \draw (4,5) node[above]{$b$};
    
    
    \draw (5,5) node[above]{$b$};
    
\draw (2.5,-1) node{$p(U')=(0,1,2,\textcolor{red}{2},1)$};
\draw (2.5,-1.7) node{$a\,a\,a\,b\,\textcolor{red}{a\,b}\,b\,a\,b\,b$};

 \end{scope}

\begin{scope}[scale=1, yshift=-9 cm]
 \fill(2cm,1.5cm); 
    \draw (0,0) grid (4,4);
    \draw[dashed,thin] (-0.5,-0.5)--(4.5,4.5);
    \draw[blue,line width=0.6mm] (0,0)--(0,1)--(1,1)--(1,3)--(2,3)--(2,4)--(4,4);
    
    
    \draw (0,1) node[left]{$a$};
    \draw (1,1) node[above left]{$b$};
    \draw (1,2) node[above left]{$a$};
    \draw (1,3) node[above left]{$a$};
    \draw (2,3) node[above left]{$b$};
    \draw (2,4) node[above]{$a$};
    \draw (3,4) node[above]{$b$};
    \draw (4,4) node[above]{$b$};

\draw (2,-1.5) node{$\zeta(p(U))=a\,b\,a\,a\,b\,a\,b\,b$};

\draw[->] (5.5,2)--(6.5,2);
\end{scope}

 \begin{scope}[scale=1,xshift= 8cm, yshift= -9.5cm]
    \fill(2cm,1.5cm); 
    \draw (0,0) grid (5,5);
    \draw[dashed,thin] (-0.5,-0.5)--(5.5,5.5);
    \draw[blue,line width=0.6mm] (0,0)--(0,1)--(1,1)--(1,3)--(2,3)--(2,4);
    \draw[blue,line width=0.6mm] (2,5)--(4,5);

    \draw[red,line width=0.6mm] 
    (2,4)--(2,5) ;
    \draw[red,line width=0.6mm] 
    (4,5)--(5,5);
    
    
    \draw (0,1) node [left]{$a$};
    
    
    \draw (1,1) node [above left]{$b$};
    
    
    \draw (1,2) node [above left]{$a$};
    
    \draw (1,3) node [above left]{$a$};
    
    \draw (2,3) node [above left]{$b$};
    
    \draw (2,4) node [above left]{$a$};
    
    
    \draw (2,5) node [above, red]{$a$};
    
    \draw (3,5) node [above]{$b$};
    
    \draw (4,5) node [above]{$b$};
    
    \draw (2,5) node{$\textcolor{red}{\bullet}$};
    
    \draw (5,5) node[above,red]{$b$};
    
     \draw (5,5) node{$\textcolor{red}{\bullet}$};

\draw (2.5,-1) node{$\zeta(p(U'))=a\,b\,a\,a\,b\,a\,\textcolor{red}{a}\,b\,b\,\textcolor{red}{b}$};

 \end{scope}
 
 \end{scope}

 \end{tikzpicture}

    \caption{Adding of an interval of level $2$ that creates a new peak in $p$.}
    \label{fig:expan}
\end{figure}
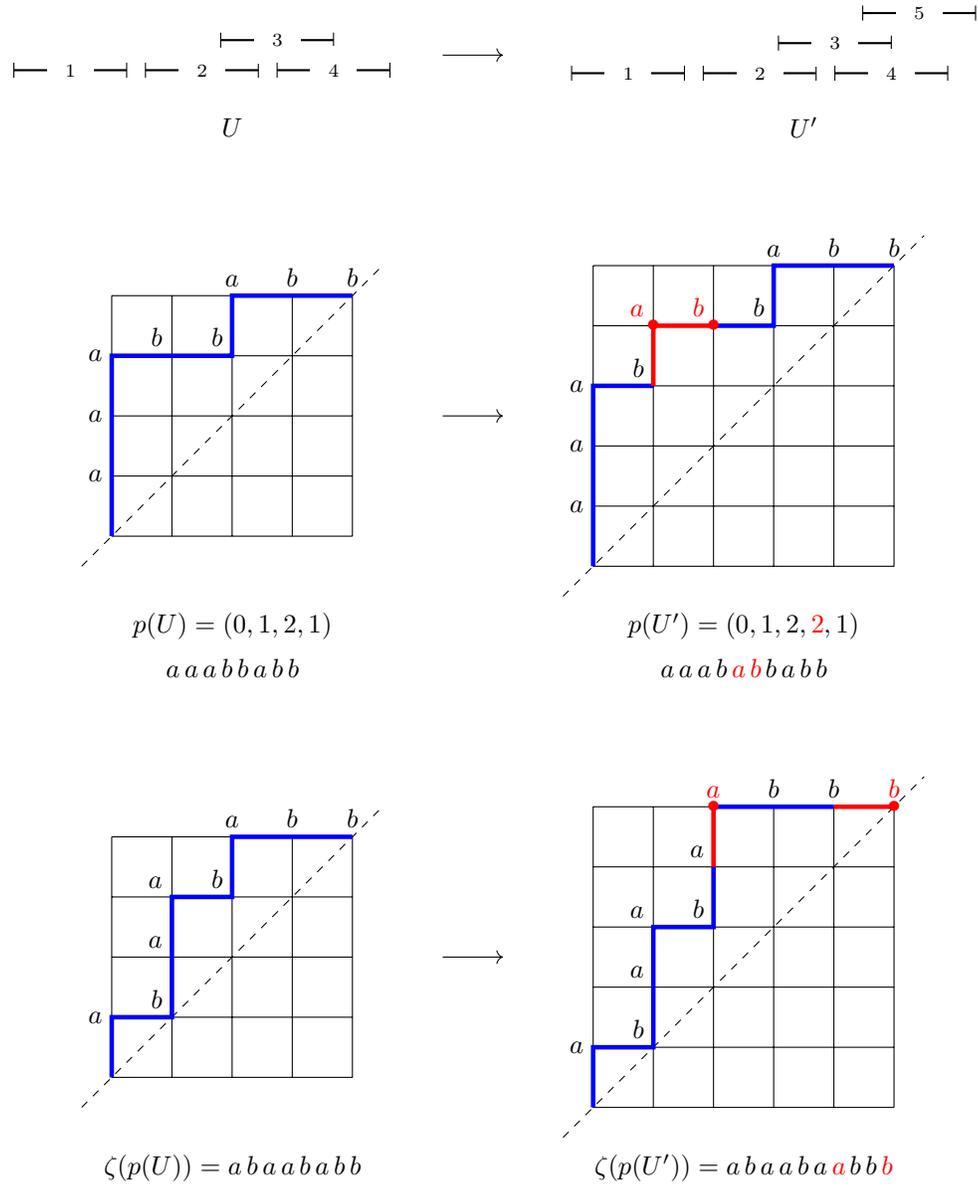

\begin{lemma}\label{deb}
Let $U\in \U_n$, and let $U'$ be obtained from it by adding an interval to the right of the intervals of $U$.
   The Dyck path
  $\zeta(p(U'))$ is obtained from $\zeta(p(U))$ by adding a final peak in position $(n-r,n+1)$, where $r$ is the sum of the number
  of occurrences of the letter $\ell$ in $p(U)$ and of the number of occurrences of the letter $\ell-1$ appearing after the position of the added letter $\ell$ in $p(U')$.
\end{lemma}

\begin{proof}
  By Lemma \ref{add}, we know that $p(U')$ is obtained from $p(U)$ by adding a final maximal peak. The height of this peak is $\ell+1$ since we added $\ell$ in the area sequence $p(U)$ to obtain this peak. The fact that this is the final maximal peak means that the peaks after are of heights smaller than $\ell+1$. 

  Let us recall that the map $\zeta$ builds a Dyck path by scanning along the lines of slope $1$ from bottom left to top right, by order of increasing height. Adding the peak of height $\ell+1$ does not change what
  we read on any height below $\ell$, nor what we read on height $\ell$ before
  we reach the right endpoint of the right step of the added peak. When we reach this right endpoint, in $\zeta(p(U'))$ we put an up step `$a$'. Then on the same height $\ell$ we read top endpoints of peaks of height $\ell$ appearing after the added peak (if any). Such peaks correspond to the letters $\ell-1$ appearing after the position of the added letter $\ell$ in $p(U')$, which will give right steps `$b$' in $\zeta(p(U'))$. Finally, we read the line of height $\ell+1$, putting a right step `$b$' in $\zeta(p(U'))$ for all maximal peaks of $p(U')$, which correspond to the occurrences of the letter $\ell$ in $p(U')$. Thus after the last up step `$a$' of $\zeta(p(U'))$ we have exactly $r+1$ right steps `$b$', so the added peak is in position $(n-r,n+1)$.
\end{proof}

\begin{lemma}
\label{lemmea}
The Dyck path $a(U')$ is obtained from $a(U)$ by adding a final peak in position $ (n-s,n+1)$, where $s$ is the number of intervals in $U'$ not comparable to 
the rightmost interval $I_{n+1}$. 
\end{lemma}

\begin{proof}
Let us recall that by definition $a(U)$ is the Dyck path whose area set is given by the boxes $(i,j)$ such that $i<j$ and $i\not\prec j$. We then add $I_{n+1}$, which is not comparable to the last $s$ intervals in $U$ (those of level $\ell$ together with a subset of those of level $\ell-1$). Thus in $U'$ we have only the new non comparable relations $i\not\prec n+1$ for $i\in\{n-s+1,\dots,n\}$, giving the corresponding boxes $(i,n+1)$ in the area set. This proves the lemma.
\end{proof}

We can now prove the main result of the paper.

\begin{Theorem} 
\label{theoremConj}
 The maps $\zeta \circ p$ and $a$ coincide. 
\end{Theorem}

\begin{proof}

We proceed by induction. We know that $a(U)=\zeta \circ p(U)$ is true for the unique $U\in \U_1$, thus the initial case is verified.

Let $n\geq 1$. Suppose that $a(U)=\zeta \circ p(U)$ is true for all $U\in \U_n$. Let $U'\in \U_{n+1}$. There exists a unit interval order poset $U\in \U_n$ such that $U'$ is obtained from $U$ by adding an $(n+1)$-st interval $I_{n+1}$ of level $\ell$ to the right of those
  of $U$. 
  
   We now establish that the number $r$ in Lemma \ref{deb} is equal to the number $s$ in Lemma \ref{lemmea}. Indeed, in $U'$ the intervals not comparable to $I_{n+1}$ are all the intervals of level $\ell$ in $U$ and a subset of those of level $\ell-1$. The latter are exactly those which correspond to the occurrences of the letter $\ell-1$ appearing after the position of the added $\ell$ in $p(U')$.
  
  Then using Lemma \ref{deb} and Lemma \ref{lemmea}, and the induction hypothesis that gives $a(U) = \zeta(p(U))$, we obtain that $a(U') = \zeta(p(U'))$, thereby finishing the induction and the proof of the theorem.
\end{proof}

\section{Graded reverse lexicographic minimality of $\tilde p(U)$} \label{grevlex}

We define the graded reverse lexicographic order on finite sequences of
$n$ non-negative integers as follows. We say that $(a_1,\dots,a_n)<(b_1,\dots,b_n)$ if $\sum_i a_i<\sum_i b_i$, or in the case that the sums are equal if $a_j>b_j$, where $j$ is the index of the first position where the two strings differ.
(Note that the inequality $a_j>b_j$ is reversed from what one might expect!
The expected inequality, $a_j<b_j$, defines graded lexicographic order.)

\begin{lemma} \label{grevlex-lemma} Let $U$ be a unit interval order.
  The graded reverse lexicographically minimal part listing for $U$ is the area
  sequence
  of a Dyck path. \end{lemma}

\begin{proof} Let $(a_1,\dots, a_n)$ be the graded reverse lexicographically minimal part
  listing for $U$. Suppose, seeking a contradiction, that $a_i>a_{i-1}+1$.
  Consider the part listing in which $a_i$ and $a_{i-1}$ have swapped
  positions. This part listing defines an isomorphic poset, and the new
  part listing is lower in graded reverse lexicographic order, so we would
  have preferred it. This is a contradiction.

  Now suppose that $a_1>0$. Consider the part listing in which $a_1$ is
  removed and $a_1-1$ is inserted at the end. This part listing produces
  an isomorphic poset, and since its sum is lower, it is lower in graded
  reverse lexicographic order, so we would have preferred it. This, too,
  is a contradiction.

  It follows that the graded reverse lexicographically minimal part listing
  for $U$ is the area sequence of a Dyck path.
  \end{proof}

\begin{corollary} For $U$ a unit interval order, the part listing $\tilde p(U)$ is the graded reverse lexicographically minimal part listing for $U$.
\end{corollary}

\begin{proof} By Lemma \ref{grevlex-lemma}, the graded reverse lexicographically
  minimal part listing for $U$ is the area sequence of some Dyck path.
  We know that $\tilde p$ defines a bijection from unit interval orders to area
  sequences of Dyck paths. Thus, there is at most one area sequence of a Dyck
  path which, when interpreted as a part listing, yields a poset isomorphic
  to $U$. The
  graded reverse lexicographically minimal part listing for $U$ is
  therefore $\tilde p(U)$.
  \end{proof}

\subsection*{Acknowledgements}

  The authors would like to thank Mathieu Guay-Paquet, Alejandro Morales,
  and Viviane Pons
  for helpful discussions.

The authors benefitted from the support of the NSERC  Discover Grants program and the Canada Research Chairs program. F.G. was supported by an NSERC USRA.

\end{document}